\documentclass[12pt]{article}     \usepackage{amsmath}
\usepackage{amsthm}               \usepackage{amsopn}
\usepackage{graphics}
\usepackage{latexsym}             \usepackage{amsfonts}
\usepackage{amssymb}
\usepackage{amsmath}
\usepackage{enumerate}
\usepackage{array}
\usepackage{graphicx}
\usepackage{amsopn}

\swapnumbers
\theoremstyle{plain}
\newtheorem{theorem}{Theorem}[section]
\newtheorem{proposition}[theorem]{Proposition}
\newtheorem{lemma}[theorem]{Lemma}
\newtheorem{corollary}[theorem]{Corollary}

\theoremstyle{definition}

\newtheorem{remark}[theorem]{Remark}

\def\qed{\hfill\rule{1ex}{1ex}\\}
\newenvironment{pf}{\noindent {\bf Proof.}}{\qed}

\title{The Mean Value Theorem and\\
Basic Properties of the Obstacle Problem for\\
Divergence Form Elliptic Operators}
\author{Ivan Blank and Zheng Hao}

\newcommand{\nc}[2]{ \newcommand{#1}{#2} }

\nc{\avint}{ {- \hspace{-3.5mm} \int} }  
\nc{\xpr}{x^{\prime}}

\nc{\R}{{\rm {I \! R}}}  
\nc{\N}{{\rm {I \! N}}}  
\newcommand{\closure}[1]{ \stackrel{\rule{.1 in}{.01 in}}{#1} }
\newcommand{\dclosure}[1]{ \stackrel{\rule{.2 in}{.01 in}}{#1} }

\newcommand{\pclosure}[1]{ \stackrel{\rule{.5 in}{.01 in}}{#1} }

\newcommand{\chisub}[1]{ {\mathbf{\chi}}_{_{#1}} }

\newcommand{\newsec}[2]{ \section{#1} \label{sec-#2}  
                         \setcounter{equation}{0}     
                         \setcounter{theorem}{0} }    

\newcommand{\refeqn}[1]{ (\!\!~\ref{eq:#1}) } 
\newcommand{\refthm}[1]{ (\!\!~\ref{#1}) }    

\nc{\Holder}{H\"{o}lder\ }

\nc{\ith}{ \ensuremath{\text{i}^{\text{th}}} }
\nc{\jth}{ \ensuremath{\text{j}^{\text{th}}} }
\nc{\kth}{ \ensuremath{\text{k}^{\text{th}}} }
\nc{\curl}{ \nabla \times }
\nc{\Div}{ \nabla \cdot }

\nc{\Ppl}{ \mathcal{M}^{+} }  \nc{\Pmn}{ \mathcal{M}^{-} }

\nc{\smiley}{ $\stackrel{\because}{\smile} \;$ }

\newcommand{\BVP}[4]{  
  \begin{equation}
        \begin{array}{rl}
           #1 & \ \text{in}
               \ \ #4 \vspace{.05in} \\
           #2 & \ \text{on} \ \ \partial #4 \;.
        \end{array}
  \label{eq:#3}
  \end{equation}    }

\newcommand{\BVPb}[3]{   \BVP{#1}{#2}{#3}{ B_{1} }  }
\newcommand{\BVPhao}[4]{  
  \begin{equation}
        \left\{
        \begin{array}{ll}
           #1 & \ \text{in}
               \ \ #4 \vspace{.05in} \\
           \ \\
           #2 & \ \text{on} \ \ \partial #4 \;.
        \end{array} \right.
  \label{eq:#3}
  \end{equation}    }

\newcommand{\BVPc}[4]{  
  \begin{equation}
        \begin{array}{rl}
           #1 & \ \text{in}
               \ \ #4 \vspace{.05in} \\
           #2 & \ \text{on} \ \ \partial #4 \;,
        \end{array}
  \label{eq:#3}
  \end{equation}    }

\newcommand{\BVPbc}[3]{   \BVPc{#1}{#2}{#3}{ B_{1} }  }

\newcommand{\BVPn}[4]{  
  \begin{equation}
        \begin{array}{rl}
           #1 & \ \text{in}
               \ \ #4 \vspace{.05in} \\
           #2 & \ \text{on} \ \ \partial #4
        \end{array}
  \label{eq:#3}
  \end{equation}    }

\newcommand{\BVPbn}[3]{   \BVPn{#1}{#2}{#3}{ B_{1} }  }

\begin{document}
\numberwithin{equation}{section}
\maketitle

\begin{abstract} \noindent
In 1963, Littman, Stampacchia, and Weinberger proved a mean value theorem for elliptic operators in
divergence form with bounded measurable coefficients.  In the Fermi lectures in 1998, Caffarelli stated
a much simpler mean value theorem for the same situation, but did not include the details of the proof.
We show all of the nontrivial details needed to prove the formula stated by Caffarelli, and in the course
of showing these details we establish some of the basic facts about the obstacle problem for general elliptic 
divergence form operators, in particular, we show a basic quadratic nondegeneracy property.
\end{abstract}

\newsec{Introduction}{Intro}
Based on the ubiquitous nature of the mean value theorem in problems involving the Laplacian, it is clear that an
analogous formula for a general divergence form elliptic operator would necessarily be very useful.  In \cite{LSW},
Littman, Stampacchia, and Weinberger stated a mean value theorem for a general divergence form operator, $L.$
If $\mu$ is a nonnegative measure on $\Omega$ and $u$ is the solution to:  \BVPc{Lu = \mu}{0}{LSWbvp}{\Omega}
and $G(x,y)$ is the Green's function for $L$ on $\Omega$
then Equation 8.3 in their paper states that $u(y)$ is equal to
\begin{equation}
    \lim_{a \rightarrow \infty}  \frac{1}{2a} \int_{a \leq G \leq 3a} u(x) a^{ij}(x) D_{x_i} G(x,y) D_{x_j} G(x,y) \; dx
\label{eq:LSWWMVT}
\end{equation}
almost everywhere, and this limit is nondecreasing.  The pointwise definition of $u$ given by this equation is
necessarily lower semi-continuous.  There are a few reasons why this formula is not as nice as the basic mean
value formulas for Laplace's equation.  First, it is a weighted average and not a simple average.  Second, it is
not an average over a ball or something which is even homeomorphic to a ball.  Third, it requires knowledge of
derivatives of the Green's function.

A simpler formula was stated by Caffarelli in \cite{C} and \cite{CR}.  That formula provides an increasing family of
sets, $D_R(x_0),$ which are each comparable to $B_R$ and such that for a supersolution to $Lu = 0$ the
average:
$$\frac{1}{|D_R(x_0)|} \int_{D_R(x_0)} u(x) \; dx$$
is nondecreasing as $R \rightarrow 0.$  On the other hand, Caffarelli did not provide any details about showing
the existence of an important test function used in the proof of this result, and showing the existence of this
function turns out to be nontrivial.  This paper grew out of an effort to prove rigorously all of the details of the
mean value theorem that Caffarelli asserted in \cite{C} and \cite{CR}.

In order to get the existence of the key test function, one must be able to solve the variational inequality or
obstacle type problem:
\begin{equation}
   D_{i} a^{ij} D_{j} V_{R} = \frac{1}{R^n} \chisub{ \{ V_{R} > 0 \} } - \delta_{x_0} 
\label{eq:yuckyObProb}
\end{equation}
where $\delta_{x_0}$ denotes the Dirac mass at $x_0.$  In \cite{CR}, the book by Kinderlehrer and Stampacchia
is cited (see \cite{KS}) for the mean value theorem.  Although many of the techniques in that book are used in the
current work, an exact theorem to give the existence of a solution to Equation\refeqn{yuckyObProb}was not found
in \cite{KS} by either author of this paper or by Kinderlehrer (\cite{K}).  The authors of this work were also unable to
find a suitable theorem in other standard sources for the obstacle problem.  (See \cite{F} and \cite{R}.)  Indeed, we
believe that without the nondegeneracy theorem stated in this paper there is a gap in the proof.

To understand the difficulty inherent in proving a nondegeneracy theorem in the divergence form case it helps to review
the proof of nondegeneracy for the Laplacian and/or in the nondivergence form case.  (See \cite{B}, \cite{BT}, and
\cite{C}.)  In those cases good use is made of the barrier function $|x - x_0|^2.$  The relevant properties are that
this function is nonnegative and vanishing at $x_0,$ it grows quadratically, and most of all, for a nondivergence form
elliptic operator $L,$ there exists a constant $\gamma > 0$ such that $L(|x - x_0|^2) \geq \gamma.$  On the other
hand, when $L$ is a divergence form operator with only bounded measurable coefficients, it is clear that
$L(|x - x_0|^2)$ does not make sense in general.

Now we give an outline of the paper.  In section two we almost get the existence of a solution to a PDE formulation
of the obstacle problem.  In section three we first show the basic quadratic regularity and nondegeneracy result
for our functions which are only ``almost'' solutions, and then we use these results to show that our ``almost''
solutions are true solutions.  In section four we get existence and uniqueness of solutions of a variational
formulation of the obstacle problem, and then show that the two formulations are equivalent.  In section five
we show the existence of a function which we then use in the sixth section to prove the mean value theorem
stated in \cite{C} and \cite{CR}, and give some corollaries.

Throughout the paper we assume that $a^{ij}(x)$ are bounded, symmetric, and uniformly elliptic,
and we define the divergence form elliptic operator
\begin{equation}
     L := D_j \; a^{ij}(x) D_i \;,
\label{eq:Ldef}
\end{equation}
or, in other words, for a function $u \in W^{1,2}(\Omega)$ and $f \in L^2(\Omega)$ we say ``$Lu = f$ in $\Omega$'' if
for any $\phi \in W_{0}^{1,2}(\Omega)$ we have:
\begin{equation}
    - \int_{\Omega} a^{ij}(x) D_{i} u D_{j} \phi = \int_{\Omega} g \phi \;.
\label{eq:Ldef2}
\end{equation}
(Notice that with our sign conventions we can have $L = \Delta$ but not $L = -\Delta.$)
With our operator $L$ we let $G(x,y)$ denote the Green's function for all of $\R^n$ and observe that the existence of
$G$ is guaranteed by the work of Littman, Stampacchia, and Weinberger.  (See \cite{LSW}.)

The results in this paper are used in a forthcoming sequel where we establish some weak regularity results for 
the free boundary in the case where the coefficients are assumed to belong to the space of vanishing mean oscillation.
The methods of that paper rely on stability, flatness, and compactness arguments.
(See \cite{BH}.)  In the case where the coefficients are assumed to be Lipschitz continuous, recent work of
Focardi, Gelli, and Spadaro establishes stronger regularity results of the free boundary.  The methods of that
work have a more ``energetic'' flavor:  They generalize some important monotonicity formulas,
and use these formulas along with the epiperimetric inequality due to Weiss and a generalization
of Rellich and N\v{e}cas' identity to prove their regularity results.  (See \cite{FGS}.)


\newsec{The PDE Obstacle Problem with a Gap}{PDEgap}

We wish to establish the existence of weak solutions to an obstacle type problem which we now describe.  We assume
that we are given
\begin{equation}
   f, a^{ij} \in L^{\infty}(B_1) \ \ \ \text{and} \ \ \ g \in W^{1,2}(B_1) \; \cap \; L^{\infty}(B_1),
\label{eq:FctSpaces}
\end{equation}
which satisfy:
\begin{equation}
   \begin{array}{l}
        \displaystyle{0 < \bar{\lambda} \leq f \leq \bar{\Lambda} \;, \rule[-.1 in]{0 in}{.3 in}} \\
        \displaystyle{a^{ij} \equiv a^{ji} \;,  \rule[-.1 in]{0 in}{.3 in}}\\
        \displaystyle{0 < \lambda |\xi|^2 \leq a^{ij} \xi_i \xi_j \leq \Lambda |\xi|^2 
           \ \ \text{for all} \ \xi \in \R^n, \ \xi \ne 0 \;, \ \ \text{and}  \rule[-.1 in]{0 in}{.3 in}} \\
        \displaystyle{g \equiv \hspace{-.15in}  \slash \; 0 \  \text{on} \ \partial B_1, \ g \geq 0. \rule[-.1 in]{0 in}{.3 in}}
   \end{array}
\label{eq:UniformEllipAndPos}
\end{equation}
We want to find a nonnegative function $w \in W^{1,2}(B_1)$ which is a weak solution of:
\BVP{Lw = \chisub{ \{ w > 0 \} }f}{w = g}{PDEverOP}{B_1}
In this section we will content ourselves to produce a nonnegative function $w \in W^{1,2}(B_1)$ which is
a weak solution of:
\BVPc{Lw = h}{w = g}{PDEverOPh}{B_1}
where we know that $h$ is a nonnegative function satisfying:
\begin{equation}
    \begin{array}{rll}
         h(x) &\!\! = 0 \ \ \ &\text{for} \ x \in \{ w = 0 \}^{\text{o}} \rule[-.1 in]{0 in}{.3 in} \\
         h(x) &\!\! = f(x) \ \ \ &\text{for} \ x \in \{ w > 0 \}^{\text{o}} \rule[-.1 in]{0 in}{.3 in} \\
         h(x) &\!\! \leq \bar{\Lambda} \ \ \ &\text{for} \ x \in \partial \{ w = 0 \} \cup \partial \{ w > 0 \} \rule[-.1 in]{0 in}{.3 in} \;,
    \end{array}
\label{eq:hprops}
\end{equation}
where for any set $S \subset \R^n,$ we use $S^{\text{o}}$ to denote its interior.  Thus $h$ agrees with
$\chisub{ \{ w > 0 \} }f$ everywhere except possibly the free boundary.
(The ``gap'' mentioned in the title to this section is the fact that we won't know that $h = \chisub{ \{ w > 0 \} }f \ a.e.$ until
we show that the free boundary (that is $\partial \{ w = 0 \} \cup \partial \{ w > 0 \}$) has measure zero.)
We will show such a $w$ exists by obtaining it as a limit of functions $w_{s}$ which are solutions to the semilinear PDE:
\BVPc{Lw = \Phi_{s}(w)f}{w = g}{PDEsemilinOP}{B_1} where for $s > 0,$ $\Phi_{s}(x) := \Phi_{1}(x/s)$
and $\Phi_{1}(x)$ is a function which satisfies

\begin{enumerate}
   \item $\Phi_{1} \in C^{\infty}(\R) \;,$
   \item $0 \leq \Phi_{1} \leq 1 \;,$
   \item $\Phi_{1} \equiv 0$ for $x < 0,$ \ \ $\Phi_{1} \equiv 1$ for $x > 1,$  and
   \item $\Phi_{1}^{\prime}(x) \geq 0$ for all $x.$ 
\end{enumerate}
The function $\Phi_{s}$ has a derivative which is supported in the interval $[0,s]$ and notice that for a fixed $x,$
$\Phi_{s}(x)$ is a nonincreasing function of $s.$

If we let $H$ denote the standard Heaviside function, but make the convention that $H(0) := 0$ then we can
rewrite the PDE in Equation\refeqn{PDEverOP}as
$$Lw = H(w)f$$
to see that it is formally the limit of the PDEs in Equation\refeqn{PDEsemilinOP}\!\!.
We also define
$$\Phi_{-s}(x) := \Phi_{s}(x + s)$$ so that we will be able to ``surround'' our solutions to our obstacle problem with
solutions to our semilinear PDEs.

The following theorem seems like it should be stated somewhere, but without further smoothness assumptions on the $a^{ij}$
we could not find it within \cite{GT}, \cite{HL}, or \cite{LU}.  The proof is a fairly standard application of the method of
continuity, so we will only sketch it.

\begin{theorem}[Existence of Solutions to a Semilinear PDE]   \label{EUSolnSemi}
Given the assumptions above, for any $s \in [-1,1] \setminus \{0\}$ there exists a $w_{s}$ which satisfies
Equation\refeqn{PDEsemilinOP}\!\!.
\end{theorem}
\begin{pf}
We provide only a sketch.  Fix $s \in [-1,1] \setminus \{0\}.$ Let $T$ be the set of $t \in [0,1]$ such that there is a
unique solution to the problem
\BVP{Lw = t\Phi_{s}(w)f}{w = g}{PDEtsemilinOP}{B_1}
We know immediately that $T$ is nonempty by observing that Theorem 8.3 of \cite{GT} shows us that $0 \in T.$
Now we need to show that $T$ is both open and closed.

As in \cite{LSW} we let $\tau^{1,2}$ denote the Hilbert space formed as the quotient space
$W^{1,2}(B_1) \slash W^{1,2}_{0}(B_1)$ and then we define the Hilbert space
\begin{equation}
   H := W^{1,2}_{0}(B_1)^{\ast} \oplus \tau^{1,2} \;,
\label{eq:HilbDirSum}
\end{equation}
where $W^{1,2}_{0}(B_1)^{\ast}$ denotes the dual space to $W^{1,2}_{0}(B_1).$
Next we define the nonlinear operator $L^{t}: W^{1,2}(B_1) \rightarrow H.$  For a
function $w \in W^{1,2}(B_1),$ we set
\begin{equation}
    L^{t}(w) = \ell^{t}(w) \oplus \mathcal{R}(w) \;,
\label{eq:GenLtDef}
\end{equation}
where $\mathcal{R}(w)$ is simply the restriction from $w$ to its boundary values in $\tau^{1,2},$ and for
any $\phi \in W^{1,2}_{0}(B_1)$ we let
\begin{equation}
    [\ell^{t}(w)](\phi) := \int_{B_1} \left( a^{ij}(x) D_{i} w D_{j} \phi + t\Phi_{s}(w) f \phi \right) \; dx \;.
\label{eq:ellmeat}
\end{equation}

In order to show that $T$ is open we need the implicit function theorem in Hilbert space.  In order to use that theorem
we need to show that the Gateaux derivative of $L^t$ is invertible.  The relevent part of that
computation is simply the observation that the Gateaux derivative of $\ell^t,$ which we denote by $D\ell^{t},$ is
invertible.  Letting $v \in W^{1,2}(B_1)$ we have
\begin{equation}
    \left[ \rule[-.1 in]{0 in}{.3 in} [D\ell^{t}(w)](\phi) \right] (v) =
    \int_{B_1} \left( a^{ij}(x) D_{i} v D_{j} \phi + t\Phi_{s}^{\prime}(w) f v \phi \right) \; dx \;.
\label{eq:GatEll}
\end{equation}
The function $d(x) := t \Phi_{s}^{\prime}(w(x)) f(x)$ is a nonnegative bounded function of $x$ and so we can apply
Theorem 8.3 of \cite{GT} again in order to verify that $L^{t}$ is invertible.

In order to show that $T$ is closed we let $t_n \rightarrow \tilde{t},$ and assume that $\{t_n\} \subset T.$  We
let $w_n$ solve
\BVPc{Lw = t_n\Phi_{s}(w)f}{w = g}{PDEtnsemilinOP}{B_1}
and observe that the right hand side of our PDE is bounded by $\bar{\Lambda}.$  Knowing this information we can use
Corollary 8.7 of \cite{GT} to conclude $||w_n||_{W^{1,2}(B_1)} \leq C,$ and we can use the 
theorems of De Giorgi, Nash, and Moser to conclude that for any $r < 1$ we have
$||w_n||_{C^{\alpha}(\closure{B_r})} \leq C.$  Elementary functional analysis allows us to conclude that a
subsequence of our $w_n$ will converge weakly in $W^{1,2}(B_r)$ and strongly in $C^{\alpha/2}(\closure{B_r})$
to a function $\tilde{w}.$  Using a simple diagonalization argument we can show that $\tilde{w}$ satisfies
\BVPc{Lw = \tilde{t}\Phi_{s}(w)f}{w = g}{PDEttildesemilinOP}{B_1}
and this fact show us that $\tilde{t} \in T.$
\end{pf}

We will also need the following comparison results:

\begin{proposition}[Basic Comparisons]  \label{BCsemiPDE}
Under the assumptions of the previous theorem and letting $w_{s}$ denote the solution to
Equation\refeqn{PDEsemilinOP}\!\!, we have the following comparison results:
\begin{enumerate}
   \item $s > 0 \ \Rightarrow w_s \geq 0 \;,$
   \item $s < 0 \ \Rightarrow w_s \geq s \;,$
   \item $t < s \ \Rightarrow w_t \geq w_s \;,$
   \item $t < 0 < s \ \Rightarrow w_s \leq w_{t} + s - t \;,$ and
   \item For a fixed $s \in [-1,1] \setminus \{0\}$ the solution, $w_{s}$ is unique.
\end{enumerate}
\end{proposition}
\begin{pf}
All five statements are proved in very similar ways, and their proofs are fairly standard, but for the
convenience of the reader, we will prove the fourth statement.  We assume that it is false, and we let
\begin{equation}
    \Omega^{-} := \{ w_{s} - w_{t} > s - t \} \;.
\label{eq:OmMinusDef}
\end{equation}
Obviously $w_s - w_{t} = s - t$ on $\partial \Omega^{-}.$
Next, observe that by the second statement we know that $\Omega^{-}$ is a subset of $\{ w_{s} > s \}.$
Thus, within $\Omega^{-}$ we have $L(w_s - w_t) = 1 - \Phi_{t}(w_{t}) \geq 0$ and so if $\Omega^{-}$
is not empty, then we contradict the weak maximum principle.
\end{pf}

We are now ready to give our existence theorem for our ``problem with the gap.''

\begin{theorem}[Existence Theorem]   \label{ExistPWG}
Given the assumptions above, there exists a pair $(w,h)$ such that $w \geq 0$ satisfies
Equation\refeqn{PDEverOPh}with an $h \geq 0$ which satisfies Equation\refeqn{hprops}\!\!. 
\end{theorem}
\begin{pf}
Using the last proposition, we can find a sequence $s_n \rightarrow 0,$ and a function $w$ such that
(with $w_n$ used as an abbreviation for $w_{s_n}$) we have strong convergence of the $w_n$ to $w$ in
$C^{\alpha}(\closure{B_r})$ for any $r < 1$ and weak convergence of the $w_n$ to $w$ in $W^{1,2}(B_1).$
Elementary functional analysis allows us to conclude that the functions $\chisub{ \{ w_n > 0 \} }f$ converge
weak-$\ast$ in $L^{\infty}(B_1)$ to a function $h$ which automatically satisfies $0 \leq h \leq \bar{\Lambda}.$
By looking at the equations satisfied by the $w_n$'s and using the convergences, it then follows very easily that
the function $w$ satisfies Equation\refeqn{PDEverOPh}\!\!, but it remains to verify that the function $h$ is
equal to $\chisub{ \{ w > 0 \} }f$ away from the free boundary.

Since the limit is continuous, the set $\{ w > 0 \}$ is already open, and by the uniform convergence of the $w_n$'s
we can say that on any set of the form $\{ w > \gamma \}$ (where $\gamma > 0$) we will have
$\Phi_{s_n}(w_n) \equiv 1$ once $n$ is sufficiently large.  Thus we must have $h = f$ on this set.  On the other hand,
in the interior of the set $\{ w = 0 \}$ we have $\nabla w \equiv 0,$ and so it is clear that in that set $h \equiv 0 \ a.e.$
\end{pf}


\newsec{Regularity, Nondegeneracy, and Closing the Gap}{OpRegNonDeg}

Now we begin with a pair $(w,h)$ like the pair given by Theorem\refthm{ExistPWG}\!\!, except that we
do not insist that it have any particular boundary data on $\partial B_1.$  In other words, in this section $w$ will
always satisfy
\begin{equation}
    L(w) = h \ \ \text{in} \ B_1,
\label{eq:InterhEqn}
\end{equation}
for a function $h$ which satisfies Equation\refeqn{hprops}\!\!.  In addition we will assume
Equations\refeqn{FctSpaces}and\refeqn{UniformEllipAndPos}hold.  By the end of this section we
will know that the set $\partial \{ w = 0\}$ has Lebesgue measure zero and so $w$ actually satisfies:
\begin{equation}
    L(w) = \chisub{ \{ w > 0 \} }f \ \ \text{in} \ B_1,
\label{eq:basicobprob}
\end{equation}
which will allow us to forget about $h$ afterward.  
Before we eliminate $h,$ we have two main results: First, $w$ enjoys a parabolic bound
from above at any free boundary point, and second, $w$ has a quadratic nondegenerate growth from such points.
It turns out that these properties are already enough to ensure that the free boundary has measure zero.

\begin{lemma} \label{parabd}
Assume that $w$ satisfies everything described above, but in addition, assume that 
$w(0)=0.$ Then there exists a $\tilde{C}$ such that
\begin{equation} \parallel w \parallel_{L^{\infty}(B_{1/2})} \; \leq \tilde{C}.
\end{equation}
\end{lemma}
\begin{proof} Let $u$ solve the following PDE:
\BVPhao{Lu = h}{u=0}{PDEboundu}{B_1}
Then Theorem 8.16 of \cite{GT} gives
\begin{equation} \parallel u \parallel_{L^{\infty}(B_{1})} \; \leq C_1.
\end{equation}
Now, consider the solution to:
\BVPhao{Lv = 0 }{v=w}{PDEboundv}{B_1}
Notice that $u(x)+v(x)=w(x)$, and in particular $0=w(0)=u(0)+v(0)$.
Then by the Weak Maximum Principle and the Harnack Inequality, we have
\begin{equation} \sup_{B_{1/2}}|v|=\sup_{B_{1/2}} v \leq C_2 \inf_{B_{1/2}} v \leq C_2 v(0) \leq C_2 (-u(0)) \leq C_2 \cdot C_1.
\end{equation}
Therefore \begin{equation} \parallel w \parallel_{L^{\infty}(B_{1/2})}\leq C
\end{equation}
\end{proof}
\begin{theorem}[Optimal Regularity]  \label{OpReg}
If $0 \in \partial \{w > 0\},$ then for any $x \in B_{1/2}$ we have
\begin{equation}
   w(x) \leq 4\tilde{C}|x|^2
\label{eq:OpReg}
\end{equation}
where $\tilde{C}$ is the same constant as in the statement of Lemma\refthm{parabd}\!\!.
\end{theorem}
\begin{proof} By the previous lemma, we know $\parallel w \parallel_{L^{\infty}(B_{1/2})} \; \leq \tilde{C}$.
Notice that for any $\gamma > 1,$
\begin{equation} u_{\gamma}(x):=\gamma^2 w \left(\frac{x}{\gamma}\right)
\end{equation}
is also a solution to the same type of problem on $B_1$, but with a new operator $\tilde{L},$ and with a new function
$\tilde{f}$ multiplying the characteristic function on the right hand side. On the other hand, the new operator has the same
ellipticity as the old operator, and the new function $\tilde{f}$ has the same bounds that $f$ had.
Suppose there exist some point $x_1 \in B_{1/2}$ such that
\begin{equation} w(x_1)> 4\tilde{C}|x_1|^2.
\end{equation}
Then since $\frac{1}{2|x_1|}>1$ and since $\frac{x_1}{2|x_1|} \in \partial B_{\frac{1}{2}}$, we have
\begin{equation} u_{\left(\frac{1}{2|x_1|}\right)}\left(\frac{x_1}{2|x_1|}\right) = \frac{1}{4|x_1|^2}w(x_1) > \tilde{C} \;,
\end{equation}
which contradicts Lemma\refthm{parabd}\!\!.
\end{proof}

Now we turn to the nondegeneracy statement.  The first thing we need is a variant of the following result from \cite{LSW}:
\begin{lemma}[Corollary 7.1 of \cite{LSW}]    \label{C71LSW}
Suppose $\mu$ is a nonnegative measure supported in $C$ which we assume is a compact subset of $B_1.$
Suppose $L$ and $\tilde{L}$ are divergence form elliptic operators exactly
of the type considered in this work, and assume that their constants of ellipticity are all contained in the
interval of positive numbers: $[\bar{\lambda}, \bar{\Lambda}].$  If
\BVPbc{Lu = \tilde{L} \tilde{u} = \mu}{u = \tilde{u} = 0}{LSW71sit}
then there exists a constant $K = K(n, C, \bar{\lambda}, \bar{\Lambda})$ such that for all $x \in C$ we have
$$ K^{-1} u(x) \leq \tilde{u}(x) \leq K u(x) \;.$$
\end{lemma}

We need to do away with the restriction that $\mu$ be supported on a compact subset of $B_1,$ but we can restrict our
attention to much simpler nonnegative measures.  In fact, the following lemma is good enough for our purposes:


\begin{lemma}   \label{LSW71Var}
Assume $L$ and $\tilde{L}$ are taken exactly as in Lemma\refthm{C71LSW}\!, and assume
\BVPb{Lw = \tilde{L} \tilde{w} = 1}{w = \tilde{w} = 0}{SimplerSit}
Then there exists a postive constant $C_0 = C_0(n, \bar{\lambda}, \bar{\Lambda})$ such that for all $x \in B_{1/4}$
we have
\begin{equation}
    C_0^{-1} w(x) \leq \tilde{w}(x) \leq C_0 w(x) \;.
\label{eq:GEFGW}    
\end{equation}
\end{lemma}
\begin{proof}
Without loss of generality we can assume that $\tilde{L}$ is the Laplacian, and we can also replace the assumption
$Lw = \Delta \tilde{w} = 1$ with the assumption $Lw = \Delta \tilde{w} = -1$ so that $w$ and $\tilde{w}$ are positive
functions.  In fact, $\tilde{w}(x) = \Theta(x)$ where we define
$$\Theta(x) := \frac{1 - |x|^2}{2n} \;.$$
It will be convenient to define the following positive universal constants:
\begin{equation}
     \theta_1 := \int_{B_1} |\nabla \Theta|^2 \ \ \ \ \ \ \ \text{and} \ \ \ \ \ \ \ \theta_2 := \int_{B_{1/2}} \Theta \;.
\label{eq:thetaconsts}
\end{equation}

Let $u$ solve
\BVPbn{Lu = -\chisub{\{B_{1/2}\}}}{u = 0}{uLSWdef} and
let $v$ solve
\BVPb{Lv = -1 + \chisub{\{B_{1/2}\}}}{v = 0}{vLSWdef}
By the strong maximum principle, both $u$ and $v$ are positive in $B_1,$ and since $w = u + v$ in $B_1,$ we
have $w > u$ in $B_1.$  By Theorem 8.18 of \cite{GT}
\begin{equation}
    \left(\frac{1}{4}\right)^{-n} ||u||_{L^1(B_{1/2})} \leq C \inf_{B_{1/4}} u \;.
\label{eq:HalfHappy}
\end{equation}
By basic facts from the Calculus of Variations, $u$ is characterized as the unique minimizer of the functional:
\begin{equation}
   J(\phi ; r) := \int_{B_1} \nabla \phi A(x) \nabla \phi - 2 \int_{B_{r}} \phi \;,
\label{eq:uJrdef}
\end{equation}
when $r$ is taken to be $1/2.$
(We are letting $A(x)$ be the matrix of coefficients for the operator $L.$)
Now we observe that for any $t > 0,$ we have
\begin{alignat*}{1}
     J(t \Theta ; 1/2) &= t^2 \int_{B_1} \nabla \Theta A(x) \nabla \Theta - 2t \int_{B_{1/2}} \Theta \\
                       &\leq t^2 \Lambda \theta_1 - 2t \theta_2 \;.
\end{alignat*}
(Recall that $\theta_1$ and $\theta_2$ are the positive universal constants defined in
Equation\refeqn{thetaconsts}above.)  Now by taking $$t := \frac{\theta_2}{\Lambda \theta_1}$$
we can conclude
\begin{equation}
    J(u ; 1/2) \leq J(t \Theta ; 1/2) \leq -\frac{\theta_2}{2 \Lambda \theta_1} =: -C_1 < 0 \;.
\label{eq:C1def}
\end{equation}
Now since $$J(u ; 1/2) \geq -2 \int_{B_1/2} u = -2 ||u||_{L^1(B_{1/2})} \;,$$
we can conclude that $$||u||_{L^1(B_{1/2})} \geq C_1/2 \;,$$
which can be combined with Equation\refeqn{HalfHappy}to get
\begin{equation}
    \inf_{B_{1/4}} w \geq \inf_{B_{1/4}} u \geq C
\label{eq:HalfDone}
\end{equation}
which is half of what we need.

On the other hand, by Theorem 8.17 of \cite{GT} we know
\begin{equation}
    \sup_{B_{1/2}} w \leq C(||w||_{L^2(B_1)} + 1) \;.
\label{eq:OthHalfHap}
\end{equation}
Using the fact that $w$ is the unique minimizer of $J( \cdot ; 1 )$ and reasoning in a fashion almost identical to
what we did above we get:
\begin{alignat*}{1}
    0 &\geq J( w ; 1 ) \\
       &\geq \lambda \int_{B_1} | \nabla w |^2 - 2 \int_{B_1} w \\
       &=\lambda ||\nabla w||^2_{L^2(B_1)} - 2 ||w||_{L^1(B_1)} \\
       &\geq C \lambda ||w||^2_{L^2(B_1)} - 2 ||w||_{L^1(B_1)} \ \ \ \ \ \ \ \ \ \text{by Poincar\'{e}'s inequality} \\
       &\geq C \lambda ||w||^2_{L^2(B_1)} - 2(||w||_{L^2(B_1)} + |B_1|)
\end{alignat*}
which forces $||w||_{L^2(B_1)} \leq C_0$ for some universal $C_0.$  Combining this equation with
Equation\refeqn{OthHalfHap}gives us what we need.
\end{proof}

\begin{lemma} \label{growth}
Let $W$ satisfy the following
\begin{equation}
    \bar{\lambda} \leq L(W) \leq \bar{\Lambda} \ \ \text{in} \ B_r \ \ \ \ \text{and} \ W \geq 0 \;,
\label{eq:grlemma1}
\end{equation}
then there exists a positive constant, $C,$ such that
\begin{equation} \sup_{\partial B_r} W \geq W(0) + Cr^2 \;.
\end{equation}
\end{lemma}
\begin{proof} Let $u$ solve
\begin{equation}
    L(u) = 0 \ \ \text{in} \ B_r \ \ \ \ \text{and} \ u=W \ \ \text{on}\ \ \partial B_r \;.
\end{equation}
Then the Weak Maximum Principle gives:
\begin{equation}
\sup_{\partial B_r} u\geq u(0).
\label{eq:wmpu}
\end{equation}
Let $v$ solve
\begin{equation}
    L(v) = L(W) \ \ \text{in} \ B_r \ \ \ \ \text{and} \ v=0 \ \ \text{on} \ \ \partial B_r \;.
\end{equation}
Notice that $v_0(x):=\frac{|x|^2-r^2}{2n}$ solves
\begin{equation}
    \Delta(v_0) = 1 \ \ \text{in} \ B_r \ \ \ \ \text{and} \ v_0=0 \ \ \text{on} \ \ \partial B_r \;.
\end{equation}
By Lemma\refthm{LSW71Var}above, there exist constants $C_1, C_2$, such that
$C_1v_0(x)\leq v(x) \leq C_2v_0(x)$ in $B_{r/4}$. In particular,
\begin{equation}
-v(0)\geq C_2\frac{r^2}{2n}.
\label{eq:lswv0}
\end{equation}
By the definitions of $u$ and $v,$ we know $W=u+v,$ therefore by Equations\refeqn{wmpu}and\refeqn{lswv0}we have
\begin{equation}
\sup_{\partial B_r} W(x) = \sup_{\partial B_r} u(x) \geq u(0) = W(0) - v(0) \geq W(0) + C_2\frac{r^2}{2n} \;.
\label{eq:NiceChain}
\end{equation}
\end{proof}
\begin{lemma} \label{posball}
Take $w$ as above, and assume that $w(0)=\gamma>0.$  Then $w>0$ in a ball $B_{\delta_0}$
where $\delta_0=C_0\sqrt{\gamma}$
\end{lemma}
\begin{proof} By Theorem\refthm{OpReg}\!\!, we know that if $w(x_0)=0$, then
\begin{equation}
   \gamma = |w(x_0)-w(0)| \leq C|x_0|^2,
\label{eq:growthold}
\end{equation}
which implies $|x_0|\geq C\sqrt{\gamma}$.
\end{proof}
\begin{lemma}[Nondegenerate Increase on a Polygonal Curve]
Let $w$ be exactly as above except that we assume that everything is satisfied in $B_2$ instead of $B_1.$
Suppose again that $w(0) = \gamma > 0,$ but now we may require $\gamma$ to be sufficiently small.  Then
there exists a positive constant, $C,$ such that
\begin{equation} \sup_{B_1} w(x) \geq C + \gamma.
\end{equation}
\end{lemma}
\begin{proof}
We can assume without loss of generality that there exists a $y \in B_{1/3}$ such that $w(y) = 0.$  Otherwise
we can apply the maximum principle along with Lemma\refthm{growth}to get:
\begin{equation}
   \sup_{B_1} w(x) \geq \sup_{B_{1/3}} w(x) \geq \gamma + C,
\label{eq:easycase}
\end{equation}
and we would already be done. 

By Lemmas\refthm{growth}and\refthm{posball}\!\!, there exist $x_1 \in \partial B_{\delta_0}$, such that
\begin{equation} w(x_1)\geq w(0)+ C\frac{\delta_0^2}{2n} = (1+C_1)\gamma
\end{equation}
For this $x_1$ and $B_{\delta_1}(x_1)$ where $\delta_1=C_0\sqrt{w(x_1)}$, Lemma\refthm{posball}guarantees the
existence of an $x_2\in \partial B_{\delta_1}(x_1)$, such that
\begin{equation} w(x_2)\geq (1+C_1)w(x_1)\geq (1+C_1)^2\gamma
\end{equation}
Repeating the steps we can get finite sequences $\{x_i\}$ and $\{\delta_i\}$ with $x_0=0$ such that
\begin{equation} w(x_i) \geq (1+C_1)^i\gamma \ \ \text{and} \ \ \delta_i = |x_{i+1} - x_i| = C_0\sqrt{w(x_i)}.
\label{eq:recurrence}
\end{equation}
Observe that as long as $x_{i} \in B_{1/3},$ because of the existence of $y \in B_{1/3}$ where $w(y) = 0$ we know that
$\delta_i \leq 2/3,$ and so $x_{i+1}$ is still in $B_1.$
Pick $N$ to be the smallest number which satisfies the following inequality:
\begin{equation}
\Sigma_{i=0}^{N}\delta_i=\Sigma_{i=0}^{N}C_0\sqrt{\gamma}(1+C_1)^{\frac{i}{2}}\geq \frac{1}{3},
\end{equation}
that is
\begin{equation}
N \geq \frac{2\ln\left[\frac{(1+C_1)^{\frac{1}{2}}-1}{3C_0\sqrt{\gamma}}+1\right]}{\ln(1+C_1)}-1.
\end{equation}
Plugging this into Equation\refeqn{recurrence}gives
\begin{align*} w(x_N)\geq & \; \gamma(1+C_1)^{\frac{2\ln\left[\frac{(1+C_1)^{\frac{1}{2}}-1}{3C_0\sqrt{\gamma}}+1\right]}{\ln(1+C_1)}-1} \\
                        = & \; \frac{\gamma}{1+C_1}\left(\frac{(1+C_1)^{\frac{1}{2}}-1}{3C_0\sqrt{\gamma}}+1\right)^2 \\
                        = & \; (\tilde{C}_0 + \tilde{C}_1 \sqrt{\gamma})^2 \\
                        \geq & \; C_2(1 + \gamma) \;,
\end{align*}
where the last inequality is guaranteed by the fact that we allow $\gamma$ to be sufficiently small.
\end{proof}

\begin{lemma} Take $w$ as above, but assume that 
$0\in \overline{\{w>0\}}$. Then
\begin{equation} \sup_{\partial B_1} w(x)\geq C.
\end{equation}
\end{lemma}
\begin{pf}
By applying the maximum principle and the previous lemma this lemma is immediate.
\end{pf}

\begin{theorem}[Nondegeneracy] \label{NonDeg}
With $C = C(n,\lambda, \Lambda, \bar{\lambda}, \bar{\Lambda}) > 0$ exactly as in the previous lemma, and if
$0 \in \pclosure{ \{w > 0\} },$ then for any $r \leq 1$ we have
\begin{equation}
    \sup_{x \in B_r} w(x) \geq Cr^2 \;.
\label{eq:NonDeg}
\end{equation}
\end{theorem}
\begin{proof} Assume there exists some $r_0 \leq 1$, such that
\begin{equation} \sup_{x \in B_{r_0}} w(x) = C_1 r_0^2< C{r_0}^2 \;.
\end{equation}
Notice that for $\gamma \leq 1$,
\begin{equation} u_{\gamma}(x):= \frac{w(\gamma x)}{\gamma^2}
\end{equation}
is also a solution to the same type of problem with a new operator $\tilde{L}$ and new function $\tilde{h}$ defined in
$B_1,$ but the new operator has the same ellipticity as the old operator, and the new $\tilde{h}$ has the same bounds
and properties that $h$ had. Now in particular for $u_{r_0}(x)=\frac{w(r_0 x)}{{r_0}^2}$, we have for any $x\in B_1$
\begin{equation} u_{r_0}(x) = \frac{w(r_0 x)}{{r_0}^2} \leq \frac{1}{r_0^2}\sup_{x \in B_{r_0}} w(x) = C_1 < C \;,
\end{equation}
which contradicts the previous lemma.
\end{proof}

\begin{corollary}[Free Boundary Has Zero Measure]  \label{FBHZM}
The Lebesgue measure of the set
$$\partial \{ w = 0 \}$$
is zero.
\end{corollary}
\begin{pf}
The idea here is to use nondegeneracy together with regularity to show that contained in any ball centered on the free
boundary, there has to be a proportional subball where $w$ is strictly positive.  From this fact it follows that
the free boundary cannot have any Lebesgue points.  Since the argument is essentially identical to the proof
within Lemma 5.1 of \cite{BT} that $\mathcal{P}$ has measure zero, we will omit it.
\end{pf}

\begin{remark}[Porosity]  \label{FBisP}
In fact, more can be said from the same argument.  Indeed, it shows that the free boundary is strongly porous
and therefore has a Hausdorff dimension strictly less than $n.$  (See \cite{M} for definitions of porosity and other
relevent theorems and references.)
\end{remark}

\begin{corollary}[Removing the ``Gap'']  \label{RTG}
The existence, uniqueness, regularity, and nondegeneracy theorems from this section and the previous section all
hold whenever
$$L(w) = h$$
is replaced by
$$L(w) = \chisub{ \{ w > 0 \} }f \;.$$
\end{corollary}


\newsec{Equivalence of the Obstacle Problems}{EEOP}

There are two main points to this section.  First, we deal with the comparatively simple task of getting existence,
uniqueness, and continuity of certain minimizers to our functionals in the relevent sets.  Second, and more importantly
we show that the minimizer is the solution of an obstacle problem of the type studied in the previous
two sections.  We start with some definitions and terminology.

We continue to assume that $a^{ij}$ is strictly and uniformly elliptic and we keep $L$ defined exactly as above.
We let $G(x,y)$ denote the Green's function for $L$ for all of $\R^n$ and observe that the existence of
$G$ is guaranteed by the work of Littman, Stampacchia, and Weinberger.  (See \cite{LSW}.)

Let
\begin{alignat*}{1}
   C_{sm,r} &:= \min_{x \in \partial B_r} G(x,0) \\
   C_{big,r} &:= \max_{x \in \partial B_r} G(x,0) \\
   G_{sm,r}(x) &:= \min \{ G(x,0), C_{sm,r} \}
\end{alignat*}
and observe that $G_{sm,r} \in W^{1,2}(B_M)$ by results from \cite{LSW} combined with the Cacciopoli Energy Estimate.
We also know that there is an $\alpha \in (0,1)$ such that $G_{sm,r} \in C^{0,\alpha}(\dclosure{B_M})$ by the
De Giorgi-Nash-Moser theorem.  (See \cite{GT} or \cite{HL} for example.)
For $M$ large enough to guarantee that $G_{sm}(x) := G_{sm,1}(x) \equiv G(x,0)$ on $\partial B_M,$ we define:
$$H_{M,G} := \{ w \in W^{1,2}(B_M) \; : \; w - G_{sm} \in W_{0}^{1,2}(B_M) \; \} $$ and
$$K_{M, G} := \{ \ w \in H_{M,G} \; : \; w(x) \leq G(x,0) \ \text{for all} \ x \in B_M \; \}. $$
(The existence of such an $M$ follows from \cite{LSW}, and henceforth any constant $M$ will be large enough so that
$G_{sm,1}(x) \equiv G(x,0)$ on $\partial B_M.$)\\

Define:
$$\Phi_\epsilon(t) := \left\{ \begin{array}{rl}
                                                   0 \ \ \ & \text{for} \ t \geq 0 \\
                                                   \ \\
                                                   -\epsilon^{-1}t \ \ \ & \text{for} \ t \leq 0 \;,
                                             \end{array} \right. $$
$$J(w, \Omega) := \int_{\Omega} (a^{ij}D_i wD_j w- 2R^{-n}w) \;, \ \ \text{and}$$
$$J_{\epsilon}(w, \Omega) := \int_{\Omega} (a^{ij}D_i wD_j w- 2R^{-n}w+2\Phi_\epsilon(G-w)) \;.$$


\begin{theorem} [Existence and Uniqueness]  \label{exun}
\begin{alignat*}{1}
   \text{Let} \ \ell_0 &:= \inf_{w \in K_{M,G}} J(w,B_M) \; \ \ \text{and} \\
   \text{let} \ \ell_{\epsilon} &:= \inf_{w \in H_{M,G}} J_{\epsilon}(w,B_M) \;.
\end{alignat*}
Then there exists a unique $w_0 \in K_{M,G}$ such that $J(w_0,B_M) = \ell_0,$
and there exists a unique $w_{\epsilon} \in H_{M,G}$ such that $J_{\epsilon}(w_{\epsilon},B_M) = \ell_{\epsilon} \;.$
\end{theorem}

\begin{pf} Both of these results follow by a straightforward application of the direct method of the Calculus of Variations.
\end{pf}
\begin{remark}   \label{RnNoGood}
Notice that we cannot simply minimize either of our functionals on all of $\R^n$ instead of $B_M$ as the Green's function is
not integrable at infinity.  Indeed, if we replace $B_M$ with $\R^n$ then
$$\ell_0 = \ell_{\epsilon} = - \infty$$
and so there are many technical problems.
\end{remark}


\begin{theorem}[Continuity] \label{contofmin} For any $\epsilon > 0,$ the function $w_{\epsilon}$ is continuous on $\dclosure{B_M}.$
\end{theorem}

\noindent
See Chapter 7 of \cite{G}.  


\begin{lemma} \label{NotBigatZ}
There exists $\epsilon >0$, $C < \infty$, such that $w_0 \leq C$ in $B_{\epsilon}.$
\end{lemma}
\begin{proof}
Let $\bar{w}$ minimize $J(w,B_M)$ among functions $w \in H_{M,G}.$
Then we have $$w_0 \leq \bar{w}.$$
Set $b := C_{big,M} = \max_{\partial B_M} G(x,0),$ and let $w_b$ minimize $J(w,B_M)$ among
$w\in W^{1,2}(B_M)$ with $$w - b \in W_{0}^{1,2}(B_M).$$
Then by the weak maximum principle, we have $$\bar{w} \leq w_b.$$
Next define $\ell(x)$ by
\begin{equation}
   \ell(x) := b+ R^{-n}\left(\frac{M^2-|x|^2}{4n}\right) \leq b+ \frac{R^{-n}M^2}{4n}<\infty.
   \label{eq:ellM1def}
\end{equation}
With this definition, we can observe that $\ell$ satisfies
\begin{alignat*}{1}
  \Delta \ell &= -\frac{R^{-n}}{2}, \ \text{in} \ B_M \ \ \ \text{and} \\
  \ell &\equiv b := \max_{\partial B_M} G \ \text{on} \ \partial B_M.
\end{alignat*}
Now let $\widetilde{\alpha}$ be $b+ \frac{R^{-n}M^2}{4n}$. By Corollary 7.1 in \cite{LSW} applied to $w_b - b$ and
$\ell - b,$ we have $$w_b \leq b + K (\ell - b) \leq b+K\widetilde{\alpha}<\infty.$$
Chaining everything together gives us $$w_0 \leq b + K\widetilde{\alpha}< \infty.$$
\end{proof}


\begin{lemma} \label{ordering}
If $0 <\epsilon_1 \leq \epsilon_2,$ then $$w_{\epsilon_1} \leq w_{\epsilon_2}.$$
\end{lemma}
\begin{proof}
Assume $0 <\epsilon_1 \leq \epsilon_2,$ and assume that $$\Omega_1 := \{ w_{\epsilon_1} > w_{\epsilon_2} \}$$
is not empty.  Since $w_{\epsilon_1} = w_{\epsilon_2}$ on $\partial B_M,$ since $\Omega_1 \subset B_M,$ and
since $w_{\epsilon_1}$ and $w_{\epsilon_2}$ are continuous functions, we know that
$w_{\epsilon_1} = w_{\epsilon_2}$ on $\partial \Omega_1.$
Then it is clear that among functions with the same data on
$\partial \Omega_1,$ \ $w_{\epsilon_1}$ and $w_{\epsilon_2}$ are minimizers of
$J_{\epsilon_1}(\cdot, \Omega_1)$ and $J_{\epsilon_2}(\cdot, \Omega_1)$
respectively.  Since we will restrict our attention to $\Omega_1$ for the rest of this proof, we will use
$J_{\epsilon}(w)$ to denote $J_{\epsilon}(w,\Omega_1).$ \\
\\
$J_{\epsilon_2}(w_{\epsilon_2}) \leq J_{\epsilon_2}(w_{\epsilon_1})$ implies
\begin{align*} & \int_{\Omega_1} a^{ij}D_i w_{\epsilon_2}D_j w_{\epsilon_2} -
                               2R^{-n}w_{\epsilon_2}+2\Phi_{\epsilon_2}(G-w_{\epsilon_2}) \\
                \leq  & \int_{\Omega_1} a^{ij}D_i w_{\epsilon_1}D_j w_{\epsilon_1} -
                               2R^{-n}w_{\epsilon_1}+2\Phi_{\epsilon_2}(G-w_{\epsilon_1}) \;,
\end{align*}
and by rearranging this inequality we get
$$\int_{\Omega_1} (a^{ij}D_i w_{\epsilon_2}D_j w_{\epsilon_2} - 2R^{-n}w_{\epsilon_2}) -
    \int_{\Omega_1} (a^{ij}D_i w_{\epsilon_1}D_j w_{\epsilon_1} - 2R^{-n}w_{\epsilon_1})$$
$$\leq \int_{\Omega_1} 2 \Phi_{\epsilon_2}(G-w_{\epsilon_1}) - 2\Phi_{\epsilon_2}(G-w_{\epsilon_2}) \;.$$
Therefore,
\begin{align*} &J_{\epsilon_1}(w_{\epsilon_2}) - J_{\epsilon_1}(w_{\epsilon_1}) \\
                     =& \int_{\Omega_1} a^{ij}D_i w_{\epsilon_2} D_j w_{\epsilon_2}
                                                     - 2R^{-n}w_{\epsilon_2}
                                                    + 2\Phi_{\epsilon_1}(G-w_{\epsilon_2}) \\
              & \ \ \ \ \ - \int_{\Omega_1} a^{ij}D_i w_{\epsilon_1} D_j w_{\epsilon_1}
                                                         - 2R^{-n}w_{\epsilon_1}
                                                        + 2\Phi_{\epsilon_1}(G-w_{\epsilon_1}) \\
                 \leq &\; 2 \int_{\Omega_1} \left[ \rule{0 in}{.15 in} \Phi_{\epsilon_2}(G-w_{\epsilon_1})
                                                                                                  - \Phi_{\epsilon_2}(G-w_{\epsilon_2}) \right] \\
              & \ \ \ \ \ - 2 \int_{\Omega_1} \left[ \rule{0 in}{.15 in} \Phi_{\epsilon_1}(G-w_{\epsilon_1})
                                                                                                   - \Phi_{\epsilon_1}(G-w_{\epsilon_2}) \right] \\
                     < &\; 0
\end{align*}
since $G-w_{\epsilon_1} < G-w_{\epsilon_2}$ in $\Omega_1$ and
$\Phi_{\epsilon_1}$ decreases as fast or faster than $\Phi_{\epsilon_2}$ decreases everywhere.
This inequality contradicts the fact that $w_{\epsilon_1}$ is the minimizer of $J_{\epsilon_1}(w)$.
Therefore, $w_{\epsilon_1} \leq w_{\epsilon_2}$ everywhere in $\Omega$.
\end{proof}


\begin{lemma} \label{w0lwep}
$w_0 \leq w_{\epsilon}$ for every $\epsilon > 0$.
\end{lemma}
\begin{proof}
Let $S:=\{w_0 > w_\epsilon\}$ be a nonempty set, let $w_1 := \min \{w_0, w_{\epsilon} \},$ and
let $w_2 := \max \{w_0, w_{\epsilon} \}.$  It follows that $w_1 \leq G$ and both $w_1$ and $w_2$
belong to $W^{1,2}(B_M).$  Since $\Phi_{\epsilon} \geq 0,$ we know that for any $\Omega \subset B_M$ we have
\begin{equation}
    J(w,\Omega) \leq J_{\epsilon}(w,\Omega)
\label{eq:epbigger}
\end{equation}
for any permissible $w.$  We also know that since $w_0 \leq G$ we have:
\begin{equation}
    J(w_0,\Omega) = J_{\epsilon}(w_0,\Omega) \;.
\label{eq:epnobiggerforw0}
\end{equation}
Now we estimate:
\begin{alignat*}{1}
J_{\epsilon}(w_1, B_M)
   &= J_{\epsilon}(w_1, S) +  J_{\epsilon}(w_1, S^c) \\
   &= J_{\epsilon}(w_{\epsilon}, S) + J_{\epsilon}(w_0, S^c) \\
   &= J_{\epsilon}(w_{\epsilon}, B_M) - J_{\epsilon}(w_{\epsilon}, S^c) + J_{\epsilon}(w_0, S^c) \\
   &\leq J_{\epsilon}(w_2, B_M) - J_{\epsilon}(w_{\epsilon}, S^c) + J_{\epsilon}(w_0, S^c) \\
   &= J_{\epsilon}(w_0, S) + J_{\epsilon}(w_{\epsilon}, S^c) - J_{\epsilon}(w_{\epsilon}, S^c) + J_{\epsilon}(w_0, S^c) \\
   &= J_{\epsilon}(w_0, S) + J_{\epsilon}(w_0, S^c) \\
   &= J_{\epsilon}(w_0, B_M) \;.
\end{alignat*}
Now by combining this inequality with Equations\refeqn{epbigger}and\refeqn{epnobiggerforw0}\!\!, we get:
$$J(w_1, B_M) \leq J_{\epsilon}(w_1, B_M) \leq J_{\epsilon}(w_0, B_M)  = J(w_0, B_M) \;,$$
but if $S$ is nonempty, then this inequality contradicts the fact that $w_0$ is the unique minimizer of $J$
among functions in $K_{M,G}.$
\end{proof}

\noindent
Now, since $w_\epsilon$ decreases as $\epsilon \rightarrow 0,$ and since the $w_\epsilon$'s are bounded from below
by $w_0,$ there exists $$\widetilde{w}= \lim_{\epsilon\rightarrow 0}w_{\epsilon}$$ and $w_0 \leq \widetilde{w}$.


\begin{lemma} \label{lsG}
With the definitions as above, $\tilde{w} \leq G$ almost everywhere.
\end{lemma}
\begin{proof}
This fact is fairly obvious, and the proof is fairly straightforward, so we supply only a sketch. \\
\\
Suppose not.  Then there exists an $\alpha > 0$ such that
$$\tilde{S} := \{ \tilde{w} - G \geq \alpha \}$$
has positive measure.  On this set we automatically have $w_{\epsilon} - G \geq \alpha \;.$
We compute $J_{\epsilon}(w_{\epsilon}, B_M)$ and send $\epsilon$ to zero.
We will get $J_{\epsilon}(w_{\epsilon}, B_M) \rightarrow \infty$ which gives us a contradiction.
\end{proof}


\begin{lemma} \label{widewisw0}
$\widetilde{w}=w_0$ in $W^{1,2}(B_M).$
\end{lemma}
\begin{proof} Since for any $\epsilon$, $w_{\epsilon}$ is the minimizer of $J_{\epsilon}(w, B_M)$, we have
\begin{align*} J_{\epsilon}(w_{\epsilon}, B_M)
     &\leq J_{\epsilon}(w_0, B_M) \\
     &\leq \int_{B_M} a^{ij}D_i w_0 D_j w_0 - 2R^{-n} w_0 + 2 \Phi_{\epsilon}(G - w_\epsilon),
\end{align*}
and after canceling the terms with $\Phi_{\epsilon}$ we have:
$$\int_{B_M} a^{ij}D_i w_{\epsilon}D_j w_{\epsilon} - 2R^{-n} w_{\epsilon}\leq \int_{B_M} a^{ij}D_i w_0 D_j w_0 - 2R^{-n} w_0.$$
Letting $\epsilon\rightarrow 0$ gives us
$$J(\tilde{w}, B_M) \leq J(w_0, B_M) \;.$$
However, by Proposition\refthm{lsG}\!\!, $\widetilde{w}$ is a permissible
competitor for the problem $\inf_{w \in K_{M,G}} J(w,B_M)$, so we have
$$J(w_0, B_M)\leq J(\tilde{w}, B_M).$$
Therefore $$J(w_0, B_M)= J(\tilde{w}, B_M),$$ and then by uniqueness, $\widetilde{w}=w_0.$
\end{proof}


\noindent
Let $W$ solve:
\BVPhao{L(w)=-\chisub{ \{ w< G \} }R^{-n}}{w=G_{sm}}{whatdef}{B_M} \\
The existence of such a $W$ is guaranteed by combining Theorem\refthm{ExistPWG}with
Corollary\refthm{RTG}\!\!.  (Signs are reversed, so to be completely precise one must apply the theorems
to the problem solved by $G-W.$)


\begin{lemma} \label{whatlG}
$W \leq G$ in $B_M$.
\end{lemma}
\begin{proof}  Let $\Omega= \{ W > G \}$ and $u:= W - G.$ Since $G$ is infinite at $0,$ and since
$W$ is bounded, and both $G$ and $W$ are continuous, we know there exists an $\epsilon > 0$
such that $\Omega \cap B_{\epsilon} = \phi.$ Then if $\Omega \neq \phi,$ then
$u$ has a positive maximum in the interior of $\Omega.$  However, since $L(W)=L(G)=0$ in $\Omega,$
we would get a contradiction from the weak maximum principle. Therefore, we have $W \leq G$ in $B_M$.
\end{proof}


\begin{lemma} \label{wtilgwhat}
$\tilde{w} \geq W$.
\end{lemma}
\begin{proof} It suffices to show $w_{\epsilon} \geq W,$ for any $\epsilon.$
Suppose for the sake of obtaining a contradiction that there exists an $\epsilon > 0$ and a point $x_0$ where
$w_{\epsilon} - W$ has a negative local minimum.
So $w_{\epsilon}(x_0) < W(x_0) \leq G(x_0).$
Let $\Omega:= \{ w_{\epsilon} < W \}$ and observe that $w_\epsilon = W$ on $\partial\Omega.$  Then $x_0$ is
an interior point of $\Omega$ and $$L(w_\epsilon)= - R^{-n} \ \ \text{in}\ \ \Omega.$$
However
\begin{equation}
     L(W - w_\epsilon)\geq -R^{-n}+R^{-n}=0 \ \ \text{in}\ \ \Omega.
\end{equation}
By the weak maximum principle, the minimum can not be attained at an interior point, and so we have a contradiction.
\end{proof}


\begin{lemma} \label{pdeminimizer}
$w_0 = \tilde{w} = W,$ and so $w_0$ and $\tilde{w}$ are continuous.
\end{lemma}
\begin{proof} We already showed that $w_0 = \tilde{w}$ in lemma\refthm{widewisw0}\!\!.
By lemma\refthm{wtilgwhat}\!\!, in the set where $W=G$, we have
\begin{equation} W=\tilde{w}=G.
\end{equation}
Let $\Omega_1:=\{ W < G \}$, it suffices to show $\tilde{w} = W$ in $\Omega_1$.
By definition of $W$, $L(W)=-R^{-n}$ in $\Omega_1$.

\noindent Using the fact that $w_0$ is the minimizer, the standard argument in the calculus of variations
leads to $L(w_0)\geq -R^{-n}.$
Therefore
\begin{equation} L(\tilde{w}-W)=L(w_0-W)\geq 0 \ \ \text{in} \ \ B_M.
\end{equation}
Notice that on $\partial \Omega_1$, $W=\tilde{w}=G$. By weak maximum principle, we have
\begin{equation} \tilde{w}=W \ \ \text{in} \ \ \Omega_1.
\end{equation}

\end{proof}


\noindent
Using the last lemma along with our definition of $W$ (see Equation\refeqn{whatdef}\!\!) we can now state the following theorem.


\begin{theorem}[The PDE satisfied by $w_0$]   \label{PDEsw0}
The minimizing function $w_0$ satisfies the following boundary value problem:
\BVPhao{L(w_0) = -\chisub{ \{ w_0 < G \} }R^{-n}}{w_0 = G_{sm}}{PDEw0}{B_M}
\end{theorem}


\newsec{Minimizers Become Independent of $M$}{MBIM}

At this point we are no longer interested in the functions from the last section, with the exception of $w_0.$
On the other hand, we now care about the dependence of $w_0$ on the radius of the ball on which it
is a minimizer.  Accordingly, we reintroduce the dependence of $w_0$ on $M,$ and so we will let $w_M$ be the
minimizer of $J(w,B_M)$ within $K(M,G),$ and consider the behavior as $M \rightarrow \infty.$  As we observed in
Remark\refthm{RnNoGood}\!\!, it is not possible to start by minimizing our functional on all of $\R^n,$ so we
have to get the key function, ``$V_R$,'' mentioned by Caffarelli on page 9 of \cite{C} by taking a limit over
increasing sets.  Note that by Theorem\refthm{PDEsw0}we know that $w_M$ satisfies
\BVPhao{L(w_M) = -\chisub{ \{G > w_M\} }R^{-n}}{w_M = G_{sm}}{PDEwM}{B_M}
The theorem that we wish to prove in this section is the following:
\begin{theorem}[Independence from $M$]   \label{indepM}
There exists $M \in \N$ such that if $M_j > M$ for $j = 1,2,$ then
$$w_{M_1} \equiv w_{M_2} \ \ \ \text{within} \ \ \ B_M$$
and
$$w_{M_1} \equiv w_{M_2} \equiv G \ \ \ \text{within} \ \ \ B_{M+1} \setminus B_M \;.$$
Furthermore, we can choose $M$ such that $M < C(n, \lambda, \Lambda) \cdot R.$
\end{theorem}
\noindent
This Theorem is an immediate consequence of the following Theorem:
\begin{theorem}[Boundedness of the Noncontact Set]   \label{FBisbdd}
There exists a constant $C = C(n, \lambda, \Lambda)$ such that for any $M \in \R$
\begin{equation}
     \{ w_M \ne G \} \subset B_{CR} \;.
\label{eq:FBbdd}
\end{equation}
\end{theorem}


\begin{proof}
First of all, if $M \leq CR,$ then there is nothing to prove.
For all $M > 1$ the function $W := G - w_M$ will satisfy:
\begin{equation}
    L(W) = R^{-n}\chisub{ \{ W > 0 \} },  \ \text{and} \ 0 \leq W \leq G \ \text{in} \ B_1^{c}.
\end{equation}
If the conclusion to the theorem is false, then
there exists a large $M$ and a large $C$ such that
$$x_0 \in FB(W) \cap \{ B_{M/2} \setminus B_{CR} \} \;.$$
Let $K := |x_0|/3.$
By Theorem\refthm{NonDeg}\!\!, we can then say that
\begin{equation}
    \sup_{B_{K}(x_0)} W(x) \geq C R^{-n} K^2 > C K^{2-n} \geq \sup_{B_{K}(x_0)} G(x)
\label{eq:nondeghere}
\end{equation}
which gives us a contradiction since $W \leq G$ everywhere.  Now note that in order to avoid the contradiction, we must have
$$C R^{-n} K^2 \leq C K^{2-n} \;,$$
and this leads to
$$K \leq C R $$
which means that $|x_0|$ must be less than $CR.$  In other words, $FB(W) \subset B_{CR}.$
\end{proof}

\noindent
At this point, we already know that when $M$ is sufficiently large, the set $\{G > w_M\}$ is contained in $B_{CR}$.
Then by uniqueness, the set will stay the same for any bigger $M$.
Therefore, it makes sense to define $w_R$ to be the solution of
\begin{equation}
    Lw = - R^{-n}\chisub{ \{ w < G \} } \ \ \ \text{in} \ \R^n
\label{eq:BasicEqn}
\end{equation}
among functions $w \leq G$ with $w = G$ at infinity.  Note that we can now obtain the function, ``$V_R$,'' that
Caffarelli uses on page 9 of \cite{C}.  The relationship is simply:
\begin{equation}
     V_R = w_R - G \;.
\label{eq:CaffVR}
\end{equation}


\newsec{The Mean Value Theorem}{MVTfinal}
Finally, we can turn to the Mean Value Theorem.

\begin{lemma}[Ordering of Sets]  \label{OoS}
For any $R < S$, we have
\begin{equation}
   \{ w_R < G \} \subset \{ w_S < G \}.
\label{eq:Inclusion}
\end{equation}
\end{lemma}
\begin{proof} Let $B_M$ be a ball that contains both $\{w_R < G \}$ and $\{ w_S < G \}.$
Then by the discussion in Section 2, we know $w_R$ minimizes
$$\int_{B_M} a^{ij} D_i wD_j w - 2wR^{-n}$$
and $w_S$ minimizes
$$\int_{B_M} a^{ij} D_i wD_j w - 2wS^{-n}.$$
Let $\Omega_1 \subset \subset B_M$ be the set $\{w_S > w_R\}.$ Then it follows that
\begin{equation}
       \int_{\Omega_1} a^{ij} D_i w_{S}D_j w_{S} - 2w_{S}S^{-n} \leq
       \int_{\Omega_1} a^{ij} D_i w_{R}D_j w_{R} - 2w_{R}S^{-n},
\label{eq:Ineq1}
\end{equation}
which implies
\begin{align*}
        \int_{\Omega_1} a^{ij} D_i w_{S}D_j w_{S} \leq
     & \int_{\Omega_1} a^{ij} D_i w_{R}D_j w_{R} + 2S^{-n}\int_{\Omega_1} (w_S-w_R) \\
 < & \int_{\Omega_1} a^{ij} D_i w_{R}D_j w_{R} + 2R^{-n}\int_{\Omega_1} (w_S-w_R).
\end{align*}
Therefore, since $w_S \equiv w_R$ on $\partial \Omega_1,$ and
\begin{equation}
        \int_{\Omega_1} a^{ij} D_i w_{S}D_j w_{S} - 2w_{S}R^{-n}
   < \int_{\Omega_1} a^{ij} D_i w_{R}D_j w_{R} - 2w_{R}R^{-n},
\label{eq:Ineq2}
\end{equation}
we contradict the fact that $w_R$ is the minimizer of $\int a^{ij} D_i wD_j w - 2wR^{-n}$.
\end{proof}

\begin{lemma}    \label{insideball}
There exists a constant $c = c(n, \lambda, \Lambda)$ such that $$B_{cR} \subset \{G > w_R\}.$$
\end{lemma}
\begin{proof}
By Lemma\refthm{NotBigatZ}we already know that there exists a constant $$C = C(n, \lambda, \Lambda)$$
such that $w_1(0) \leq C.$  Then it is not hard to show that
\begin{equation}
     \| w_1 \|_{L^{\infty}(B_{1/2})} \leq \tilde{C}.
\label{eq:busc}
\end{equation}
By \cite{LSW} for any elliptic operator $L$ with given $\lambda$ and $\Lambda$, we have
\begin{equation}
    \frac{c_1}{|x|^{n-2}} \leq G(x) \leq \frac{c_2}{|x|^{n-2}}.
\label{eq:BasicGreenEst}
\end{equation}
By combining the last two equations it follows that there exists a constant $c = c(n, \lambda, \Lambda)$ such that
$$B_{c}\subset\{G > w_1\}.$$
It remains to show that this inclusion scales correctly.  \\
\\
\noindent
Let $v_R := G - w_R$ (so $v_R = - V_R$).  Then $v_R$ satisfies
\begin{equation}
Lv_R = \delta - R^{-n}\chisub{ \{ v_R > 0\} } \ \ \text{in} \ \R^n \;.
\end{equation}
Now observe that by scaling our operator $L$ appropriately, we get an operator $\tilde{L}$
with the same ellipticity constants as $L,$ such that
\begin{equation}
     \tilde{L} \left( R^{n-2}v_R(Rx) \right) = \delta - \chisub{ \{ v_{_R}(Rx)>0 \} } \;.
\label{eq:tilLeqn}
\end{equation}
So we have
$$B_c \subset \left\{ x \; \rule[-.08 in]{.01 in}{.25 in}  \; v_R(Rx) > 0 \right\},$$
which implies
\begin{equation}
     B_{cR} \subset \left\{ \rule[-.08 in]{0 in}{.25 in} v_R(x) > 0 \right\}.
\label{eq:winningeqn}
\end{equation}
\end{proof}


\noindent Suppose $v$ is a supersolution to 
$$Lv = 0,$$
i.e. $L v \leq 0.$ Then for any $\phi \geq 0,$ we have 
\begin{equation}
  \int_{\Omega} vL\phi \leq 0 .
  \label{eq:supersolution}
\end{equation}
If $R < S,$ then we know that $w_R \geq w_S,$ and so the function $\phi = w_R - w_S$ is a permissible test function.
We also know:
\begin{equation} 
   L\phi = R^{-n}\chisub{ \{ G > w_R \} } - S^{-n}\chisub{ \{ G > w_S \} }.
  \label{eq:zheng}
\end{equation} 
By observing that $v \equiv 1$ is both a supersolution and a subsolution and by plugging in our $\phi,$ we arrive at
\begin{equation}
   R^{-n}|\{G > w_R\}| = S^{-n}|\{G > w_S\}|,
  \label{eq:hao}
\end{equation}
and this implies
\begin{equation}
   L\phi = C\left[\frac{1}{|\{G > w_R\}|}\chisub{ \{ G > w_R \} } - \frac{1}{|\{G > w_S\}|}\chisub{ \{ G > w_S \} } \right].
  \label{eq:zhenghao}
\end{equation}
Now, Equation\refeqn{supersolution}implies 
\begin{equation} 
   0 \geq \int_{\Omega} vL\phi = C\left[\frac{1}{|\{G > w_R\}|}\int_{\{G > w_R\}} v -  \frac{1}{|\{G > w_S\}|}\int_{\{G > w_S\}} v \right].
   \label{eq:haozheng}
\end{equation}
Therefore, we have established the following theorem:
\begin{theorem} [Mean Value Theorem for Divergence Form Elliptic PDE]
Let $L$ be any divergence form elliptic operator with ellipticity $\lambda$, $\Lambda$.
For any $x_0 \in \Omega$, there exists an increasing family $D_R(x_0)$ which satisfies the following:
\begin{enumerate}
  \item $B_{cR}(x_0) \subset D_R(x_0) \subset B_{CR}(x_0)$, with $c,$ $C$ depending only on $n,$ $\lambda$ and $\Lambda$.
  \item For any $v$ satisfying $L v \geq 0$ and $R<S$, we have
       \begin{equation}
              v(x_0) \leq \frac{1}{|D_R(x_0)|}\int_{|D_R(x_0)|} v \leq  \frac{1}{|D_S(x_0)|}\int_{D_S(x_0)} v .
       \label{eq:MVTres}
       \end{equation}
\end{enumerate}
\end{theorem}

\noindent
As on pages 9 and 10 of \cite{C}, (and as Littman, Stampacchia, and Weinberger already observed using their own
mean value theorem,) we have the following corollary:
\begin{corollary}[Semicontinuous Representative]  \label{SemiRep}
Any supersolution $v,$ has a unique pointwise defined representative as
\begin{equation}
   v(x_0) := \lim_{R \downarrow 0} \frac{1}{|D_R(x_0)|}\int_{|D_R(x_0)|} v(x) dx \;.
\label{eq:ptrep}
\end{equation}
This representative is lower semicontinuous:
\begin{equation}
   v(x_0) \leq \lim_{x \rightarrow x_0} v(x)
\label{eq:ptreplsc}
\end{equation}
for any $x_0$ in the domain.
\end{corollary}

\noindent
We can also show the following analogue of G.C. Evans' Theorem:
\begin{corollary}[Analogue of Evans' Theorem]   \label{EvTh}
Let $v$ be a supersolution to $Lv = 0,$ and suppose that $v$ restricted to the support of $Lv$ is
continuous.  Then the representative of $v$ given by Equation\refeqn{ptreplsc}is continuous.
\end{corollary}
\begin{pf}
This proof is almost identical to the proof given on pages 10 and 11 of \cite{C} for $L = \Delta.$
\end{pf}


\newsec{Acknowledgements}{Ack}
We are endebted to Luis Caffarelli, Luis Silvestre, and David Kinderlehrer for very useful conversations.
We also wish to thank the anonymous referees for very careful readings and for helpful suggestions.


\bibliographystyle{abbrv}
\bibliography{simple}

\end{document}